\documentclass[a4paper]{amsart}

\usepackage{amsfonts,amssymb,mathtools}
\usepackage{textcomp, color}

\theoremstyle{plain}

\newtheorem*{main}{Main Theorem}

\newtheorem{thm}{Theorem}[section]
\newtheorem{lem}[thm]{Lemma}

\newtheorem{cor}[thm]{Corollary}

\theoremstyle{definition}

\newtheorem{rmk}[thm]{Remark}

\begin{document}

\title[Beauville $p$-groups of wild type]{Beauville $p$-groups of wild type and groups of maximal class} 

\author[G.A.\ Fern\'andez-Alcober]{Gustavo A.\ Fern\'andez-Alcober}
\address{Department of Mathematics\\ University of the Basque Country UPV/EHU\\
48080 Bilbao, Spain}
\email{gustavo.fernandez@ehu.eus}

\author[N.\ Gavioli]{Norberto Gavioli}
\address{Universit\`a degli Studi dell'Aquila (Italy)}
\email{gavioli@univaq.it}

\author[\c{S}.\ G\"ul]{\c{S}\"ukran G\"ul}
\address{Department of Mathematics\\ Middle East Technical University\\
06800 Ankara, Turkey}
\email{gsukran@metu.edu.tr}

\author[C.M.\ Scoppola]{Carlo M.\ Scoppola}
\address{Universit\`a degli Studi dell'Aquila (Italy)}
\email{scoppola@univaq.it}

\thanks{The first and third authors acknowledge financial support from the Spanish Government, grant MTM2014-53810-C2-2-P, and from the Basque Government, grant IT974-16.}

\begin{abstract}
Let $G$ be a Beauville finite $p$-group.
If $G$ exhibits a `good behaviour' with respect to taking powers, then every lift of a Beauville structure of $G/\Phi(G)$ is a Beauville structure of $G$.
We say that $G$ is a Beauville $p$-group of wild type if this lifting property fails to hold.
Our goal in this paper is twofold: firstly, we fully determine the Beauville groups within two large families of $p$-groups of maximal class, namely metabelian groups and groups with a maximal subgroup of class at most $2$; secondly, as a consequence of the previous result, we obtain infinitely many Beauville $p$-groups of wild type.
\end{abstract}

\maketitle

\section{Introduction}

\emph{Beauville groups\/} are finite groups that arise in the construction of an interesting class of complex surfaces, the so-called \emph{Beauville surfaces\/}, introduced by Catanese in \cite{cat} following an example of Beauville
\cite[page 159]{bea}.
The problem of determining what finite groups are Beauville has received considerable attention in the past years, see the survey paper \cite{jon}.
Catanese \cite{cat} showed that a finite abelian group is a Beauville group if and only if it is isomorphic to
$C_n\times C_n$, with $n>1$ and $\gcd(n,6)=1$.
A remarkable result, proved independently by Guralnick and Malle \cite{GM} and by Fairbairn, Magaard and Parker \cite{FMP}, is that every non-abelian finite simple group other than $A_5$ is a Beauville group.

The determination of Beauville groups is based on the following purely group-theoretical characterization.
Given a group $G$ and $S=\{x,y\}\subseteq G$, the triple $T=\{x,y,xy\}$ is uniquely determined by $S$ up to conjugacy, and so we can define
\begin{equation}
\label{Sigma}
\Sigma(S)
=
\bigcup_{t\in T} \, \langle t \rangle^g.
\end{equation}
Then $G$ is a Beauville group if and only if it is a $2$-generator group and
\begin{equation}
\label{BS}
\text{$\Sigma(S_1) \cap \Sigma(S_2)=1$ for some $2$-element sets of generators $S_1$ and $S_2$ of $G$.}
\end{equation}
We say that $S_1$ and $S_2$ form a \emph{Beauville structure\/} for $G$.

Knowledge about Beauville finite $p$-groups was scarce until very recently, and only groups of small order or some very specific families had been tested for being Beauville (see \cite{BBF} and \cite{bos}).
In \cite[Theorem 2.5]{FG}, Fern\'andez-Alcober and G\"ul extended Catanese's criterion in the specific case of $p$-groups from abelian groups to a much wider family, including all $p$-groups that are known to have a `good behaviour' with respect to taking powers, and in particular all groups of order at most $p^{\hspace{.4pt}p}$.
More precisely, if $G$ is such a $2$-generator group of exponent $p^e$, then
\begin{equation}
\label{good power structure}
\text{$G$ is a Beauville group if and only if $p\ge 5$ and $|G^{p^{e-1}}|\ge p^2$.}
\end{equation}
This criterion is not valid for all finite $p$-groups, and in \cite[Corollary 2.12]{FG} there are infinitely many examples (among which infinitely many $p$-groups of maximal class with an abelian maximal subgroup) for which
(\ref{good power structure}) fails to hold.

Different Beauville structures in the same group can give rise to non-isomorphic Beauville surfaces, and another main problem in this theory is to determine the number of Beauville surfaces arising from a given Beauville group.
Gonz\'alez-Diez, Jones, and Torres-Teigell \cite{GJT} have given a formula for this number in the case of abelian Beauville groups.
The first step to solve this problem is to determine all Beauville structures in the group, which is relatively easy for abelian groups, but much more difficult in the general setting, due to the effect of conjugacy in (\ref{Sigma}).
In any case, it is interesting to obtain as many Beauville structures as possible in a Beauville group.
The finite $p$-groups considered in \cite[Theorem 2.5]{FG} have a good behaviour in this respect, since they have the property that \emph{every lift of a Beauville structure of $G/\Phi(G)\cong C_p\times C_p$ is a Beauville structure
of $G$\/}.
We say that a Beauville $p$-group $G$ is of \emph{tame type\/} if this property holds in $G$; otherwise, we say that $G$ is of \emph{wild type\/}.
Thus all $p$-groups with good power structure are of tame type but, as a careful analysis shows, all counterexamples to (\ref{good power structure}) given in \cite[Corollary 2.12]{FG} are also of tame type.

Our goal in this paper is to determine all Beauville groups within two very large families of $p$-groups of maximal class, namely those which are either metabelian or have a maximal subgroup of class at most $2$.
These families have been studied in great detail by Miech \cite{mie} and by Leedham-Green and McKay \cite{LM,LM2}, respectively; in particular, constructions are given in these papers showing that there is a large number of such groups.
Thus we provide a huge extension of Corollary 2.12 of \cite{FG}, which only covers one $p$-group of maximal class of every order.
An important consequence of our results is that they provide an infinite family of Beauville $p$-groups of wild type.

Next we state our main theorem in this paper.
The Beauville $p$-groups of maximal class of order at most $p^{\hspace{.4pt}p}$ were already determined in
\cite[Corollary 2.10]{FG}: they are exactly the groups of exponent $p$, with $p\ge 5$.
Thus we restrict to groups of order at least $p^{\hspace{.4pt}p+1}$.
In the theory of groups of maximal class, a significant role is played by the maximal subgroup
$G_1=C_G(G'/\gamma_4(G))$, and this is also the case in our result.
In the remainder, we use the term \emph{maximal branch\/} in a $2$-generator finite $p$-group $G$ to mean the difference $B(M)=M\smallsetminus\Phi(G)$, where $M$ is a maximal subgroup of $G$.

\begin{main}
Let $G$ be a $p$-group of maximal class of order $p^n$, with $n\ge p+1$, and assume that either $G$ is metabelian or $G$ contains a maximal subgroup of class at most $2$.
Then $G$ is a Beauville group if and only if $p\ge 5$ and one of the following two cases holds, where $X$ is the set of all elements of $G\smallsetminus\Phi(G)$ of order $p$:
\begin{enumerate}
\item
$X=G\smallsetminus G_1$.
\item
$X$ is the union of exactly two maximal branches of $G$, and either $n\not\equiv 2 \pmod{p-1}$, or $n=p+1$ and one of those maximal branches is $B(G_1)$.
\end{enumerate}
Also, all groups in (i) are Beauville groups of tame type, all groups in (ii) are of wild type, and there exist infinitely many groups in each of the cases.
\end{main}

\vspace{5pt}

\textit{Notation.\/}
If $G$ is a finitely generated group, $d(G)$ denotes the minimum number of generators of $G$.
On the other hand, if $G$ is a finite $p$-group, we write $\Omega_i(G)$ for the subgroup generated by all elements of $G$ of order at most $p^i$.

\section{Proof of the main theorem}

%\begin{dfn}
%Let $G$ be a $2$-generator finite $p$-group.
%If $M$ is a maximal subgroup of $G$, we say that the subset
%$B(M)=M\smallsetminus \Phi(G)$ is a \emph{maximal branch\/} of $G$.
%Furthermore, we say that $B(M)$ is a maximal branch of
%\emph{uniform \textcolor{red}{(another name? we also speak of uniform elements)} order\/} $p^k$ if all of its elements have order $p^k$.
%\end{dfn}

In this section $G$ will always denote a group of order $p^n$ and class $n-1$, where $n\ge 2$.
These groups are called \emph{$p$-groups of maximal class\/}.
The Beauville groups of order $p^2$ and $p^3$ can be easily determined, so we assume that $n\ge 4$ in the remainder.
Then we set $G_i=\gamma_i(G)$ for all $i\ge 2$, and $G_1=C_G(G_2/G_4)$.
We have $|G_i:G_{i+1}|=p$ for all $i=1,\ldots,n-1$, and the \emph{degree of commutativity\/} of $G$ is defined as the maximum (non-negative) integer $\ell\le n-3$ such that $[G_i,G_j]\le G_{i+j+\ell}$ for all $i,j\ge 1$.
Also, the only normal subgroups of $G$ are $G$ itself, its maximal subgroups, and the subgroups $G_i$ for $i\ge 2$.
Observe that $Z(G)=G_{n-1}$, $\Phi(G)=G_2$, and $d(G)=2$.

In his seminal paper \cite{bla}, Blackburn established many of the cornerstones of the theory of groups of maximal class.
More precisely, he showed that $\ell>0$ if and only if $G_1=C_G(G_{n-2})$, and that this is the case whenever
$n\ge p+2$.
An element $s\in G$ is called \emph{uniform\/} if $s\notin G_1\cup C_G(G_{n-2})$.
Then $s^p\in Z(G)$ and the order of $s$ is at most $p^2$.
Also, $|C_G(s)|=p^2$ and the conjugates of $s$ are exactly the elements in the coset $s\Phi(G)$.
If we choose an element $s_1\in G_1\smallsetminus G_2$ and we define $s_i=[s_1,s,\overset{i-1}{\ldots},s]$, then $s_i\in G_i\smallsetminus G_{i+1}$ for all $i=1,\ldots,n-1$. 
Thus $Z(G)=\langle s_{n-1} \rangle$.
Finally, we recall a couple of facts about the power structure of $G$.
If $|G|\le p^{\hspace{.4pt}p+1}$ then $\exp G/Z(G)=\exp \Phi(G)=p$, while for $|G|\ge p^{\hspace{.4pt}p+2}$ we have
\begin{equation}
\label{pth powers}
x^p \in G_{i+p-1} \smallsetminus G_{i+p},
\quad
\text{for all $x\in G_i\smallsetminus G_{i+1}$ and all $i=1,\ldots,n-p$.}
\end{equation}
These properties of $p$-groups of maximal class  can be found in \cite{gus}, \cite[Chapter 4, Section 14]{hup}, or
\cite[Chapter 3]{LM3}.

\begin{lem}
\label{all branches uniform}
Let $G$ be a $p$-group of maximal class, and let $M$ be a maximal subgroup of $G$.
Then all elements in $B(M)$ have the same order, and if $M\ne G_1$, this order is either $p$ or $p^2$.
\end{lem}

\begin{proof}
Assume first that $M\ne G_1$, $C_G(G_{n-2})$.
If $s\in B(M)$ then we have $B(M)=\cup_{i=1}^{p-1} \, s^i\Phi(G)$.
All powers $s^i$ are uniform elements of $G$ for $i=1,\ldots,p-1$, and they have the same order $p$ or $p^2$.
Since all elements in the coset $s^i\Phi(G)$ are conjugate to $s^i$, the result follows in this case.

Now suppose that $M=G_1$ or $C_G(G_{n-2})$.
If $n\ge p+2$ then $M=G_1$ and by (\ref{pth powers}), every element in $B(M)$ has order $p^k$, where $k=\lceil \frac{n-1}{p-1} \rceil$.
If $n\le p+1$ then $M$ is a regular $p$-group in which $\Phi(G)$ is a maximal subgroup of exponent $p$.
Thus $\Omega_1(M)=\Phi(G)$ or $M$, and all elements of $B(M)$ have order $p$ or $p^2$, respectively.
\end{proof}

If $G$ is a $p$-group of maximal class, we denote by $\mu(G)$ the number of maximal subgroups
$M\ne G_1$ for which all elements in $B(M)$ have order $p$.
The next result shows that the value of $\mu(G)$ is quite restricted if either $G$ is metabelian or $G_1$ is of class at most $2$.
Its proof relies on two technical results of Miech regarding the calculation of $p$th powers in groups of maximal class that can be found in \cite{mie} and \cite{mie2}.

\begin{thm}
\label{values of mu}
Let $G$ be a $p$-group of maximal class.
If either $G$ is metabelian or $G_1$ is of class at most $2$, then $\mu(G)=0$, $1$, $2$ or $p$.
Furthermore, if $G$ has an abelian maximal subgroup then $\mu(G)=0$, $1$, or $p$.
\end{thm}

\begin{proof}
If $p=2$ then $G$ has only $2$ maximal subgroups different from $G_1$ and the result is obviously true.
Thus we consider $p$ to be an odd prime in the remainder of the proof.

Let us assume that $\mu(G)\ge 2$ and prove that $\mu(G)=2$ or $p$.
Since $\Phi(G)=G_2$, we can choose a uniform element $s$ and an element $s_1\in G_1\smallsetminus G_2$ so that $s$ and $ss_1$ belong to two maximal branches with all elements of order $p$.

If $G$ is metabelian, then by \cite[Lemma~8]{mie} and since $s^p=1$, we can write $(ss_1^i)^p=a^ib^{i^2}$ for all $i=1,\ldots,p-1$, where
\[
a = s_1^ps_2^{\binom{p}{2}}\ldots s_p^{\binom{p}{p}}
\]
and $b$ is a power of $s_{n-1}$ independent of $i$.
Since $(ss_1)^p=1$, we have $a=b^{-1}$ and consequently
\[
(ss_1^i)^p = b^{i^2-i} = b^{i(i-1)}.
\]
Since $b$ is of order at most $p$, this power is $1$ for some $i=2,\ldots,p-1$ (observe that $p$ is odd) if and only if $b=1$.
It follows from Lemma \ref{all branches uniform} that $\mu(G)$ is either $2$ or $p$.

Now assume that $G_1$ is of class at most $2$.
By applying \cite[Theorem~4]{mie2}, we get
\begin{equation}
\label{miech result 2}
(ss_1^i)^p = s_1^{ip} \sigma_{i,1}^{\binom{p}{2}}\ldots \sigma_{i,p-1}^{\binom{p}{p}} x_{i,p},
\end{equation}
where $\sigma_{i,k} = [s_1^i,s,\overset{k}{\ldots},s]$ and
\begin{equation}
\label{Q(p)}
x_{i,p} = \prod_{k=1}^{p-1}\prod_{\ell=0}^{k-1} [\sigma_{i,k},\sigma_{i,\ell}]^{B(p,k,\ell)},
\end{equation}
for some integer $B(p,k,\ell)$ depending on $p$, $k$ and $\ell$, but not on $i$.

Let us show, by induction on $k$, that
\begin{equation}
\label{sigma_i,k}
\sigma_{i,k} = s_{k+1}^i t_{k+1}^{\binom{i}{2}},
\quad
\text{for some $t_{k+1} \in [G_1,G_1]\le Z(G_1)$.}
\end{equation}
First of all, observe that the basis of the induction is given by
\[
\sigma_{i,1} = [s_1^i,s] = s_1^{-i} (s_1^s)^i = s_1^{-i}(s_1s_2)^i = s_2^i[s_2,s_1]^{\binom{i}{2}},
\]
where the last equality follows from the condition that the class of $G_1$ is at most $2$.
Now if $k\ge 2$, by the induction hypothesis we have
\begin{align*}
\sigma_{i,k}
&=
[s_k^i,s] [t_k^{\binom{i}{2}},s]
=
[s_k^i,s] [t_k,s]^{\binom{i}{2}}
=
s_{k+1}^i [s_{k+1},s_k]^{\binom{i}{2}} [t_k,s]^{\binom{i}{2}}
\\
&=
s_{k+1}^i \big( [s_{k+1},s_k] [t_k,s] \big)^{\binom{i}{2}},
\end{align*}
and the induction is complete.

Now from (\ref{sigma_i,k}) we get
\[
[\sigma_{i,k},\sigma_{i,\ell}]
=
[s_{k+1}^i,s_{l+1}^i]
=
[s_{k+1}, s_{l+1}]^{i^2},
\]
and this implies that $x_{i,p}=x_{1,p}^{i^2}$ for all $i=1,\ldots,p-1$.
Hence (\ref{miech result 2}) yields
\[
(ss_1^i)^p
=
\big( s_1^ps_2^{\binom{p}{2}}\dots s_p \big)^i
\Big(
t_2^{\binom{p}{2}}\dots t_p^{\binom{p}{p}}
\,
\prod_{j=1}^{p-1} \prod_{k=j+1}^{p}\, [s_j,s_k]^{\binom{p}{j}\binom{p}{k}} 
\Big)^{\binom{i}{2}}
\
x_{1,p}^{i^2}.
\]
Since $p$ is odd, we can write again $(ss_1^i)^p=a^ib^{i^2}$ for all $i=1,\ldots,p-1$, with $a$ and $b$ not depending on $i$.
As above, we conclude that $\mu(G)=2$ or $p$ also in this case.

Finally, if $G$ has an abelian maximal subgroup then obviously that subgroup must be $G_1$.
Observe that in the above discussion we have $b\in [G_1,G_1]$.
Hence $b=1$ in the present case and $(ss_1^i)^p=1$ for all $i=1,\ldots,p-1$.
In other words, if $\mu(G)\ge 2$ then necessarily $\mu(G)=p$.
This completes the proof.
\end{proof}

%\textcolor{red}{Can we give an example of a $p$-group of maximal class for which $2<\mu(G)<p$?}

We also need the following lemma about maximal subgroups of class $\le 2$ in a group of maximal class.

\begin{lem}
\label{max of class 2}
Let $G$ be a $p$-group of maximal class of order at least $p^5$.
If $G$ has a maximal subgroup $M$ of class at most $2$ then $M=G_1$.
\end{lem}

\begin{proof}
Suppose for a contradiction that $M\ne G_1$.
By definition, we have $G_1=C_G(G_2/G_4)$, and consequently the commutator subgroup $[M,G_2]$ is not contained in $G_4$.
Since the only normal subgroups of $G$ contained in $G_2$ are of the form $G_i$, it follows that $[M,G_2]=G_3$.
On the other hand, by \cite[Theorem 4.6]{gus}, the degree of commutativity of $G/G_5$ is positive.
As a consequence, we have $[G_1,G_3]\le G_5$ i.e.\ $G_1=C_G(G_3/G_5)$.
Arguing as above, this implies that $[M,G_3]=G_4$.
Now since $|M:G_2|=p$, we have $M'=[M,G_2]$ and then $\gamma_3(M)=[M,G_2,M]=G_4\ne 1$, taking into account that $|G|\ge p^5$.
This contradicts the assumption that the class of $M$ is at most $2$.
\end{proof}

Now we can proceed to the proof of our main theorem.

\begin{thm}
\label{main}
Let $G$ be a $p$-group of maximal class of order $p^n$, with $n\ge p+1$, and assume that either $G$ is metabelian or $G$ contains a maximal subgroup of class at most $2$.
Then $G$ is a Beauville group if and only if $p\ge 5$ and one of the following two cases holds:
\begin{enumerate}
\item
$\mu(G)=2$, and either $n\not\equiv 2 \pmod{p-1}$, or $n=p+1$ and one of the maximal branches consisting of elements of order $p$ is $B(G_1)$.
\item
$\mu(G)=p$.
\end{enumerate}
In the first case, all Beauville groups are of wild type, and in the second case, all Beauville groups are of tame type.
\end{thm}

\begin{proof}
We first prove the `only if' part of the statement.
Let us suppose that $S_1$ and $S_2$ form a Beauville structure for $G$, with associated triples $T_1$ and $T_2$.
We start by showing that $p\ge 5$.
By way of contradiction, assume that $p=2$ or $3$.
Then $G$ has positive degree of commutativity \cite[Theorem 4.6]{gus}, and consequently all elements in
$G\smallsetminus G_1$ are uniform.
Now the elements in $T_1$, and similarly those in $T_2$, lie in different maximal branches of $G$.
Since $G$ has at most $4$ maximal subgroups, some $x_1\in T_1$ and $x_2\in T_2$ lie in the same maximal branch $B(M)$, where $M\ne G_1$.
Now observe that for every $x\in B(M)$, all elements in the coset $x\Phi(G)$ are conjugate to $x$.
Hence some power of $x_1$ is conjugate to some power of $x_2$.
This is contrary to the condition $\Sigma(S_1)\cap \Sigma(S_2)=1$, which holds according to (\ref{BS}).
Thus we have $p\ge 5$ and, in particular, $n\ge 6$.
By Lemma \ref{max of class 2}, in the case where $G$ has a maximal subgroup of class $\le 2$, that subgroup must be $G_1$.

Let us prove that $\mu(G)\ge 2$.
Otherwise there exist $x_1\in T_1$ and $x_2\in T_2$ which are uniform elements of order $p^2$.
It follows that $\langle x_1^p\rangle=Z(G)=\langle x_2^p\rangle$, again a contradiction.
By applying Theorem \ref{values of mu}, we conclude that either $\mu(G)=2$ or $\mu(G)=p$.

Assume now that $\mu(G)=2$, and either that $n=k(p-1)+2$ for some $k\ge 2$ or that $n=p+1$ and the maximal branch $B(G_1)$ consists of elements of order $p^2$.
Again, we seek a contradiction.
In this case, each of the triples $T_1$ and $T_2$ contains an element which lies either in $B(G_1)$ or in a maximal branch $B(M)$ consisting of elements of order $p^2$, with $M\ne G_1$.
Let $x$ be any element of $B(G_1)$.
If $n=k(p-1)+2$ for some $k\ge 2$ then we have $x^{p^k}\in Z(G)\smallsetminus 1$, by (\ref{pth powers}).
On the other hand, if $n=p+1$ and the elements of $B(G_1)$ are of order $p^2$, then since $\exp G/Z(G)=p$,
we have $x^p\in Z(G)\smallsetminus 1$.
In any case, we get $Z(G)\subseteq \Sigma(S_1)\cap \Sigma(S_2)$, which is contrary to our assumption that $S_1$ and $S_2$ form a Beauville structure.
This completes the proof of the first implication.

Let us prove the converse.
Thus we assume that $p\ge 5$ in the remainder.
Suppose first that $\mu(G)=p$.
In this case, we have to show that $G$ is a Beauville group of tame type, i.e.\ that every lift to $G$ of a Beauville structure of $G/\Phi(G)$ yields a Beauville structure of $G$.
Let us consider then two minimal sets of generators $S_1$ and $S_2$ of $G$ that map onto a Beauville structure of
$G/\Phi(G)$.
Then at most one of the elements in $T_1\cup T_2$ lies in $B(G_1)$.
Hence for every choice of $x_1\in T_1$ and $x_2\in T_2$, at least one of the elements, say $x_1$, is of order $p$.
Thus if $\langle x_1^g \rangle$ and $\langle x_2^h \rangle$ have non-trivial intersection for some $g,h\in G$, then
$\langle x_1^g \rangle \subseteq \langle x_2^h \rangle$ and $\langle x_1\Phi(G) \rangle=\langle x_2\Phi(G) \rangle$.
This is impossible, since $x_1\Phi(G)$ and $x_2\Phi(G)$ participate in different triples of a Beauville structure of
$G/\Phi(G)$.
Hence $S_1$ and $S_2$ form a Beauville structure of $G$.

Now suppose that we are in case (i), and choose a uniform element $s$ and an element
$s_1\in G_1\smallsetminus G_2$ such that both $s$ and $ss_1$ are of order $p$.
We claim that $S_1=\{s,s_1\}$ and $S_2=\{ss_1^2,ss_1^4\}$ is a Beauville structure of $G$.
We need to see that $\langle x_1^g \rangle \cap \langle x_2^h \rangle=1$ for all $x_1\in T_1$, $x_2\in T_2$ and $g,h\in G$.
Observe that $x_1^g$ and $x_2^h$ lie in different maximal subgroups of $G$, since $p\ge 5$.
Thus if $x_1=s$ or $ss_1$, or if $x_1=s_1$ and $B(G_1)$ consists of elements of order $p$, we can argue as in the previous paragraph.
Hence we assume that $x_1=s_1$ and that $|G|=p^n$ with $n\ge p+2$ and $n\ne k(p-1)+2$ for all $k\ge 2$.
By applying repeatedly (\ref{pth powers}), we have $\langle s_1 \rangle \cap Z(G)=1$.
Since $\Omega_1(\langle x_2 \rangle)=Z(G)$ for all $x_2\in T_2$, the claim follows.
Thus $G$ is a Beauville group.

Let us finally see that $G$ is of wild type in case (i).
There are at least two maximal branches $B(M_1)$ and $B(M_2)$ which consist of elements of order $p^2$.
It is always possible to construct a Beauville structure in $G/\Phi(G)\cong C_p\times C_p$ in which the first set of generators contains an element from $M_1/\Phi(G)$ and the second set an element from $M_2/\Phi(G)$.
However, no lift of this structure can be a Beauville structure of $G$, since
$\langle x_1^p \rangle=\langle x_2^p \rangle=Z(G)$ for all $x_1\in B(M_1)$ and $x_2\in B(M_2)$.
\end{proof}

\begin{rmk}
\label{always valid}
Actually, the `if part' of the previous theorem is valid for all $p$-groups of maximal class of order at least
$p^{\hspace{.4pt}p+1}$, without requiring that $G$ is metabelian or that $G$ has a maximal subgroup of class at most $2$.
\end{rmk}

The case of groups of maximal class with an abelian maximal subgroup is especially easy to describe.

\begin{cor}
Let $G$ be a $p$-group of maximal class of order at least $p^{\hspace{.4pt}p+1}$.
If $G$ has an abelian maximal subgroup, then $G$ is a Beauville group if and only if every element of
$G\smallsetminus G_1$ is of order $p$.
All these Beauville groups are of tame type.
\end{cor}

\begin{proof}
This follows immediately by combining Theorems \ref{values of mu} and \ref{main}.
\end{proof}

We remark that the only infinite pro-$p$ group of maximal class $P$ has an abelian subgroup $A$ of index $p$ with the property that all elements in $P\smallsetminus A$ are of order $p$.
Thus by taking finite quotients of $P$ we get infinitely many examples of Beauville groups of tame type in our main theorem. 
The existence of infinitely many groups of wild type follows from the construction of metabelian $p$-groups of maximal class given by Miech in \cite{mie}.
More precisely, if we consider metabelian groups satisfying the condition $[G_1,G_2]=G_{n-p+2}$, then Lemma 8 of \cite{mie} together with our proof of Theorem \ref{values of mu} shows the existence of infinitely many groups with
$\mu(G)=2$.
If $n\not\equiv 2 \pmod{p-1}$ then these groups are Beauville of wild type, according to Theorem \ref{main}.

\vspace{10pt}

Every quotient of order $\ge p^2$ of a $p$-group of maximal class is obviously again of maximal class.
We conclude with the following surprising consequence of our main theorem.

\begin{cor}
Let $G$ be a $p$-group of maximal class, where $p\ge 5$.
Then every proper quotient of $G$ is a Beauville group, and it is of tame type.
\end{cor}

\begin{proof}
It suffices to show that $G/Z(G)$ is a Beauville group.
If $|G/Z(G)|\le p^{\hspace{.4pt}p}$ then, since $\exp G/Z(G)=p$, we only need to apply Corollary 2.10 of \cite{FG}, 
Assume now that $|G/Z(G)|\ge p^{\hspace{.4pt}p+1}$.
Since the $p$th powers of all uniform elements of $G$ lie in $Z(G)$, it follows that $\mu(G/Z(G))=p$.
Now the result follows from Theorem \ref{main} and Remark \ref{always valid}.
\end{proof}

\end{document}